\documentclass[10pt]{amsart}

\usepackage{latexsym, amsmath,amssymb}
\usepackage{color}

\renewcommand{\epsilon}{\varepsilon}
\newcommand{\newsection}[1]
{\subsection{#1}\setcounter{theorem}{0} \setcounter{equation}{0}
\par\noindent}

\newtheorem{theorem}{Theorem}

\newtheorem{lemma}[theorem]{Lemma}
\newtheorem{corr}[theorem]{Corollary}

\newtheorem{proposition}[theorem]{Proposition}
\newtheorem{deff}[theorem]{Definition}

\newcommand{\bth}{\begin{theorem}}
\newcommand{\ble}{\begin{lemma}}
\newcommand{\bcor}{\begin{corr}}

\newcommand{\bdeff}{\begin{deff}}

\newcommand{\bprop}{\begin{proposition}}
\newcommand{\ele}{\end{lemma}}
\newcommand{\ecor}{\end{corr}}
\newcommand{\edeff}{\end{deff}}

\newcommand{\eprop}{\end{proposition}}

\newcommand{\supp}{\text{supp}}
\renewcommand{\Pi}{\varPi}

\renewcommand{\epsilon}{\varepsilon}

\newcommand{\tidle}{\tilde}

\newcommand{\K}{{\mathcal K}}
\newcommand{\R}{{\mathbb R}}

\newcommand{\la}{{\langle}}
\newcommand{\ra}{{\rangle}}
\newcommand{\cd}{{\,\cdot\,}}
\newcommand{\bdy}{{\partial\K}}
\newcommand{\ext}{{\R^4\backslash\K}}
\renewcommand{\S}{{\mathbb{S}}}
\renewcommand{\O}{{\mathcal{O}}}
\newcommand{\N}{{\mathbb{N}}}

\begin{document}
\bibliographystyle{plain}

\title[Global existence of quasilinear wave equations]
{
Global existence for systems of quasilinear wave equations
in $(1+4)$-dimensions
}

\author{Jason Metcalfe}
\author{Katrina Morgan}
\address{Department of Mathematics, University of North Carolina, Chapel Hill}
\email{metcalfe@email.unc.edu, katri@live.unc.edu}


\begin{abstract}
H\"ormander proved global existence of solutions for sufficiently
small initial data for scalar wave equations in $(1+4)-$dimensions of the form $\Box u =
Q(u, u', u'')$ where $Q$ vanishes to second order and $(\partial_u^2
Q)(0,0,0)=0$.  Without the latter condition, only almost global
existence may be guaranteed.  The first author and Sogge considered
the analog exterior to a star-shaped obstacle.  Both results relied on
writing the lowest order terms $u\partial_\alpha u = \frac{1}{2}\partial_\alpha u^2$ and as
such do not immediately generalize to systems.  The current study
remedies such and extends both results to the case of multiple speed systems. 
\end{abstract}

\maketitle


\newsection{Introduction}
In \cite{MS4}, the authors study a quasilinear wave equation in four
spatial dimensions whose nonlinearity $Q(u,\partial u, \partial^2 u)$
vanishes to second order and
is subject to the restriction
\begin{equation}
 \label{restriction}
(\partial_u^2 Q)(0,0,0)=0,
\end{equation}
which has the effect of disallowing a term of the form $u^2$.  The
global existence result of \cite{Hormander} is extended to equations
exterior to star-shaped obstacles with Dirichlet boundary conditions.
In \cite{MS4}, the first author inserted a careless
comment that the techniques carry over with trivial change to the case
of systems of equations.  This, unfortunately, is not the case as the
terms of the form $u\partial u$ were written as $\frac{1}{2} \partial
u^2$ in order to apply \cite[Proposition 4]{MS4}.  The current study
seeks to remedy this by providing techniques that do suffice to show
global existence for systems of equations of the same type.  The
original study \cite{Hormander} on $\R\times \R^4$ employed a similar
technique as \cite{MS4}, and as such, the extension to systems appears
to be new even in the $(1+4)$-dimensional boundaryless case.

Let us more specifically introduce the problem at hand.  We shall
examine systems of quasilinear equations of the form
\begin{equation}\label{generalEquation}
  \begin{cases}
    \Box_{c_I} u^I = Q^I(u,u',u''),\quad (t,x)\in \R_+\times\ext,\quad
    I=1,2,\dots,M,\\
    u^I(t,\cd)|_\bdy = 0,\\
    u^I(0,\cd)=f^I,\quad \partial_t u^I(0,\cd)=g^I
  \end{cases}
\end{equation}
where $\Box_{c_I}=\partial_t^2-c_I^2 \Delta$ is the d'Alembertian with
speed $c_I$.  We shall, without loss of generality, assume that
\begin{equation}\label{speeds}0<c_1\le c_2\le \dots \le c_M.\end{equation}
  Here $u=(u^1,\dots, u^M)$.  We use $u'$ and $\partial u$ to denote the space-time
gradient $(\partial_t u, \nabla_x u)$, but reserve $\nabla u =\nabla_x
u$ to denote the spatial gradient. 
The obstacle $\mathcal{K}$ is taken to be compact, to have smooth
boundary, and to be star-shaped (with respect to the origin).  By
translating and 
scaling, without loss of generality, we may take $0\in \K\subset \{|x|<1\}$.

The nonlinear
term $Q$ is smooth in its arguments, vanishes to second order, and
satisfies \eqref{restriction}.  Moreover,
it has the form
\begin{equation}\label{Q}
Q^I(u,u',u'')=A^I(u,u')+ 
B^{IJ,\alpha\beta}(u,u')\partial_\alpha\partial_\beta u^J.
\end{equation}
Here $A$ is
taken to vanish to second order at $(u,u')=(0,0)$, and $B$ vanishes to
first order at the origin.  We also assume the symmetry condition
\begin{equation}
 \label{symmetry}
B^{IJ,\alpha\beta}(u,u')=B^{JI,\alpha\beta}(u,u')=B^{IJ,\beta\alpha}(u,u'),\quad
1\le I, J\le M, \, 0\le \alpha,\beta\le 4.
\end{equation}
The above uses the convenient notation $\partial_0=\partial_t$.  Here
and throughout this
paper, we shall utilize the summation convention where repeated
indices are summed.  Greek indices $\alpha, \beta, \gamma$ are summed from $0$ to the spatial
dimension $4$, while lower case Latin indices $i, j, k$ are implicitly summed from $1$ to
$4$.  Repeated uppercase $I, J, K$ will be implicitly summed from $1$
to $M$.   We shall reserve $\mu$, $\nu$, and $\sigma$ to denote multiindices.

To solve \eqref{generalEquation}, one must assume that the data
satisfy some compatibility conditions.  These are well known, and we
shall only tersely describe them.  A more detailed exposition is
available in, e.g., \cite{KSS2}.  We write $J_ku=\{\partial^\mu_x
u\,:\, 0\le |\mu|\le k\}$.  For any formal $H^m$ solution $u$, we
can express $\partial_t^k u(0,\cd)=\psi_k(J_kf,J_{k-1}g)$, $0\le k\le m$
for some compatibility functions $\psi_k$.  The compatibility
condition of order $m$ for data $(f,g)\in H^m\times H^{m-1}$ simply
requires that $\psi_k$ vanishes on $\bdy$ for $0\le k\le m-1$.  For
$(f,g)\in C^\infty$, we say that the data satisfy the compatibility
condition to infinite order if the above holds for all
$m$. 

\begin{theorem}\label{main_thm}
Suppose $\K\subset \{x\in \R^4 : |x|<1\}$ is star-shaped with respect
to the origin and has smooth boundary.  Suppose $Q$ satisfies
\eqref{restriction}, \eqref{Q}, and \eqref{symmetry}.  Assume that
$(f,g)\in (C^\infty(\ext))^{2M}$ satisfy the compatibility conditions to
infinite order, have components that are supported in $(\ext)\cap \{|x|<R_0\}$, and 
\begin{equation}\label{smallness}
\sum_{|\mu|\le N} \|\nabla^\mu f\|_{L^2(\ext)} + \sum_{|\mu|\le N}
\|\nabla^\mu g\|_{L^2(\ext)} = \varepsilon.
\end{equation}
Then if $N\ge 12$ and $\varepsilon>0$ is sufficiently small, \eqref{generalEquation}
has a unique global solution.  
\end{theorem}

In the above theorem, we have taken the data to be compactly supported
for convenience and for clarity of exposition.  It is highly likely
that the same result holds for sufficiently small and regular data that decay
sufficiently fast at infinity.  Moreover, the assumed regularity is not
nearly sharp.  The methods employed will not approach the sharp
threshold.

Such quasilinear wave equations in $\R_+\times \R^4$  whose nonlinearity depends only
on derivatives of the solution were known to enjoy global existence
for sufficiently small data \cite{klainermanSob}.  Dependence on the
solution rather than its derivatives does not mesh as well with the
energy methods that are typically applied.  In this case, 
\cite{Hormander} shows that without the hypothesis \eqref{restriction}
almost global existence, which means that the lifespan grows
exponentially as the size of the data shrinks, is what is possible.
Moreover, \cite{Hormander} continues and demonstrates that
\eqref{restriction} suffices to establish small data global existence
for scalar equations.  To the authors' knowledge, however, the current
result for systems is new even in the boundaryless context.

Following the work \cite{KSS}, which first demonstrated the utility
of local energy estimates as a means to proving long time existence
for wave equations in exterior domains, a number of studies tackled the
existence of solutions to quasilinear wave equations with small initial data in
exterior domains when the nonlinearity depends only on derivatives of
the solution.  See., e.g., \cite{KSS3}, \cite{mns2, mns1}, \cite{MS3,
  MS2} for some of the most general results.  Only having the local
energy estimate for the flat wave equation, which fails to account for
the geometric perturbations introduced by the quasilinear
nonlinearities, requires additional
techniques in order to control the highest order energies.
  When, however, the
obstacle is assumed to be star-shaped, due to the existence of a local
energy estimate for sufficiently small time-dependent perturbations of
the wave equation, simpler techniques, which are more direct analogs
of \cite{KSS}, are available.  See \cite{MS2, ms_mathz}. 

Like in the boundaryless case, nonlinear dependence on the solution
rather than only its derivatives complicates the situation.  The works
\cite{DZ}, \cite{DMSZ}, and \cite{ZZ} explore the general case outside of
star-shaped obstacles and prove analogs
of the lifespans in the boundaryless case of \cite{Lindblad},
\cite{Hormander}, and \cite{LiZ}, \cite{LS} respectively.  See \cite{HM4d, HM3d} for
similar results in more general geometries.  The subsequent work
\cite{MS4} examines scalar equations subject to \eqref{restriction} exterior to $4$-dimensional
star-shaped obstacles as in the boundaryless case of \cite{Hormander}.

As mentioned previously, \cite{Hormander} and \cite{MS4} rely on
writing the terms of the form $u\partial u$ in divergence form
$\frac{1}{2}\partial (u^2)$ for which better estimates are available.
For systems of equations, however, it cannot be guaranteed that all
such first order terms are in divergence form.  The key new idea is to
use a lemma from \cite{MTT} that allows a derivative to be exchanged
for additional decay within the light cone.

We work outside of star-shaped obstacles so that the techniques of
\cite{ms_mathz} are available.  As in \cite{HM4d, HM3d}, it is
expected that this geometric condition could be relaxed
significantly.

We shall rely on a variant of Klainerman's method of invariant vector
fields \cite{Klainerman, klainermanSob} as was adapted to exterior domains in
\cite{KSS, KSS3}, \cite{MS3}, and subsequent works.  In particular,
to avoid having vector fields with unbounded normal components on $\bdy$, a restricted set
of vector fields is employed.  The generators of space-time
translations, spatial rotations, and scaling will be used.  We set
\[\Gamma= (\partial_0, \partial_1, \partial_2, \partial_3, \Omega_{12}, \Omega_{13}, \Omega_{14}, \Omega_{23},
  \Omega_{24}, \Omega_{34}, S)\]
with 
\[\Omega_{ij} = x_i \partial_j - x_j\partial_i,\quad S=t\partial_t +
  r\partial_r.\]
It is worth noting that $[\partial, \Gamma]u = \O(|\partial u|)$.
We will abbreviate multi-index notation and, for any $N\in \N$, denote
\[\Gamma^{\le N} u = \sum_{|\mu|\le N} \Gamma^{\mu} u,\quad
\partial^{\le N} u = \sum_{|\mu|\le N} \partial^\mu u.\]

We end this section by fixing additional notations.  We will denote
$\Box u = (\Box_{c_I} u^I)$.  And we will fix $\beta(\rho)$ to be a
smooth cutoff that is identically $1$ for $\rho \le 1$ and $0$ for
$\rho > 2$.  Furthermore, we let $\beta_R(\rho)= \beta(\rho/R)$.

We shall use $S_t = [0,t]\times \R^4$ and $S_t^\K = [0,t]\times
(\ext)$ to denote space-time strips in Minkowski space and the
exterior domain respectively.  We shall be working in mixed norms.
When two Lebesgue spaces appear, it is understood to be a space-time
norm:
\[\|u\|_{L^pL^q} = \Bigl(\int_0^t \Bigl(\int_\ext |u(s,x)|^q\,dx\Bigr)^{p/q}\,ds\Bigr)^{1/p}.\]
When three occurrences appear, this indicates that the norms are in
$t, r,$ and $\omega$ with the convention
\[\|u\|_{L^pL^qL^\theta} = \Bigl(\int_0^t \Bigl[\int_0^\infty \Bigl(\int_{\S^3}
  |u(s,r\omega)|^\theta\,d\sigma(\omega)\Bigr)^{q/\theta}\, r^3\,
  dr\Bigr]^{p/q}\,ds\Bigr)^{1/p}.\] 
Due to the appearance of cutoff functions in the sequel, we shall only
need this latter norm in the boundaryless case, which is what is
stated here.  We shall use $L^p_{2^k}$ to indicate that the norm is
restricted to an annulus (or a ball in the case of $k=0$): $c2^k < \la
\cd\ra < C2^k$.  The norm may be in the 
time or the spatial variables.  The positive constants $c, C$ will be 
permitted to change from line to line but may not depend on any
important parameters in the problem.

This article is organized as follows.  In the next section, the main
linear estimates are collected.  These are variants of energy and
local energy estimates.  The first ones are local energy estimates for
the solution without a derivative and are obtained by appropriately
dividing through by a derivative and using a variant of a Sobolev
embedding.  These are collected from \cite{DMSZ}, \cite{HMSSZ}, and
\cite{MS4}.  The second class of estimates apply to small,
time-dependent perturbations of $\Box$ and, as stated, are from
\cite{ms_mathz} but are heavily influenced by the preceding works
\cite{Sterbenz}, \cite{MS}.  The third section includes the main decay
estimate, which is from \cite{klainermanSob} and yields decay in $|x|$
at the cost of admissible vector fields.  Here we gather an estimate
of \cite{MTT} that allows us to exchange a derivative for additional
decay when sufficiently within the light cone.  The final section is
devoted to the proof of Theorem~\ref{main_thm}.


\newsection{Energy and local energy estimates}
In this section, the linear $L^2$ based estimates, which are
variants of energy and local energy bounds, are collected.  For the
first several estimates, as we will be cutting away from the obstacle,
it will suffice to consider bounds in Minkowski space with vanishing
initial data.  The basic uniform energy bound and local energy
estimate states that solutions to the scalar equation
\begin{equation}
  \label{bdyless}
  \begin{cases}
    \Box_c w = F,\quad (t,x)\in \R_+\times\R^4,\\
    w(0,\cd)= \partial_t w(0,\cd)=0.
  \end{cases}
\end{equation}
satisfy
\begin{equation}
  \label{LE}
  \|w'(t,\cd)\|_{L^2(\R^4)} + \|\la x\ra^{-1/2-} w'\|_{L^2L^2(S_t)}
  \lesssim \|F\|_{L^1L^2(S_t) + \la x\ra^{-1/2-} L^2L^2(S_t)}.
\end{equation}
Here and throughout, these estimates are global in nature, and the
implicit constants are independent of $t$.  
See, e.g., \cite{MST} for some history and more general results.

As the nonlinearity under consideration allows for dependence on the
solution rather than only its derivatives, variants of the above shall
be required where the solution is estimated rather than only its
derivatives.  To that end, we record the following previously
established results.  See, also, \cite{FW}, \cite{LMSTW}, and \cite{MW} for some closely related estimates.
\begin{proposition}[{\cite[Theorem 2.3]{DMSZ}, \cite[Lemma 3.1]{HMSSZ}}]
  \label{propWeightedStrichartz}
Let $w$ be a smooth solution to \eqref{bdyless} where $\supp\,
F(t,\cd)\subset \{|x|>1\}$ for all $t\ge 0$.  Then, for $t>0$,
\begin{equation}
  \label{lenod}
\|w(t,\cd)\|_{L^2(\R^4)} \\\lesssim \||x|^{-1}
F\|_{L^1 L^1 L^2(S_t)},
\end{equation}
and
for
$0<\gamma<\frac{1}{2}$, we have 
\begin{equation}
  \label{weightedStrichartz}
\||x|^{-\frac{1}{2}-\gamma} w\|_{L^2L^2(S_t)} \lesssim
\||x|^{-1-\gamma} F\|_{L^1 L^1 L^2(S_t)}.
\end{equation}
\end{proposition}

The first estimate follows by dividing through by a derivative and
applying a Sobolev-type estimate that is akin to, e.g.,
\cite[(3.20a)]{Sideris}.  Specifically, see the dual version
\cite[Lemma 2.2]{DMSZ}.  The second estimate is established by proving
a variant of \eqref{LE} that relates the power of the weight to the
regularity (see \cite[(3.6)]{HMSSZ}) and combining such with a trace lemma.

The next result will allow us to easily handle the commutators that
will appear when we cut off away from the obstacle in order to apply
Proposition \ref{propWeightedStrichartz}.  This
is essentially \cite[(2.12)]{DMSZ}, \cite[Proposition 3]{MS4}.
\begin{proposition}
  Let $\tidle{w}$ be a smooth solution to \eqref{bdyless}, and suppose that $F(t,x)=0$ for $|x|>2R_0$.  Then,
  \begin{equation}
    \label{cpctF}
   \|\tilde{w}(t,\cd)\|_{L^2(\R^4)} + \|\la x\ra^{-1/2-} \tilde{w}\|_{L^2L^2(S_t)} \lesssim \|F\|_{L^2L^2(S_t)}.
  \end{equation}
\end{proposition}
Here, as above, the constant is independent of $t$, but it may depend
on $R_0$.

In order to prove this result, \eqref{LE} is applied to
$w=\Delta^{-1}\partial_j \tilde{w}$.  Since the kernel of
$\Delta^{-1}\partial_j$ is $\O(|x-y|^{-3})$, Young's inequality
completes the proof.

A final such result shows that better bounds are available on the
solution when the nonlinearity is in divergence form.  The following
is based on ideas of \cite{Hormander}.
\begin{proposition}[{\cite[Proposition 4]{MS4}}]
Let $v$ be a smooth solution to
\begin{equation}
  \label{divFormeq}
  \begin{cases}
    \Box v = a_\alpha \partial_\alpha G,\quad (t,x)\in \R_+\times
    \R^4,\\
v(0,\cd)=\partial_t v(0,\cd)=0,
  \end{cases}
\end{equation}
where $a_\alpha$ are constants.  Moreover, assume that $G(0,\cd)\equiv
0$.  Then
\begin{equation}
  \label{divFormLE}
  \|\la x\ra^{-1/2-\delta} v\|_{L^2L^2(S_t)} +
  \|v(t,\cd)\|_{L^2(\R^4)} \lesssim \int_0^t \|G(s,\cd)\|_{L^2(\R^4)}\,ds.
\end{equation}
\end{proposition}
Here if $\Box v_1=G$ with vanishing data, it is noted that $v =
 a_\alpha \partial_\alpha v_1$.  The result then follows from \eqref{LE}.

Due to the quasilinear nature of the study, we shall require variants
of \eqref{LE} that account for the geometry introduced by the
nonlinearity.  In particular, we need a version of \eqref{LE} that
holds for small, time-dependent perturbations of $\Box$.  To that end,
we define
\[(\Box_h u)^I = (\partial_t^2-c_I^2\Delta)u^I + 
  h^{IJ,\alpha\beta}(t,x)\partial_\alpha\partial_\beta u^J,\]
and we consider
\begin{equation}
  \label{pert}
\begin{cases}
  \Box_h u = F+G,\quad (t,x)\in \R_+\times \ext,\\
u|_{\bdy} = 0,\\
u(0,\cd)=f,\quad \partial_t u(0,\cd)=g.
\end{cases}
\end{equation}
We assume that the perturbation satisfies the symmetry conditions
\begin{equation}\label{sym}h^{IJ,\alpha\beta}=h^{JI,\alpha\beta}=h^{IJ,\beta\alpha}
\end{equation}
and the smallness condition
\begin{equation}\label{smallpert} |h| := \sum_{I,J=1}^M \sum_{\alpha,\beta = 0}^4
  |h^{IJ,\alpha\beta}|<\delta\ll 1.
\end{equation}
We also set the notation
\[|\partial h| = \sum_{I,J=1}^M \sum_{\alpha,\beta,\gamma=0}^4
  |\partial_\gamma h^{IJ,\alpha\beta}|.\]

\begin{proposition}[{\cite[Theorem 2.4]{ms_mathz}}]\label{kss_prop}
Suppose $\K$, as above, satisfies $0\in \mathcal{K}\subset \{x\in \R^4\,:\, |x|<1\}$,
is star-shaped with respect to the origin, and has a smooth boundary.
Assume that the $h^{IJ,\alpha\beta}$ satisfy \eqref{sym} and
\eqref{smallpert} for $\delta>0$ sufficiently small.  If $G(s,\cd)$ is
supported where $|x|<2$ for every $s\in [0,t]$, if $f, g$ vanish for $|x|>R_0$, and if
$u$ is a smooth solution to \eqref{pert} that vanishes for
large $|x|$ for each $s\in [0,t]$, then
\begin{multline}
  \label{kss}
  \|\la x\ra^{-1/2-}\Gamma^{\le N} u'\|_{L^2L^2(S^\K_t)} +
  \|\Gamma^{\le N} u'(t,\cd)\|_{L^2(\ext)}
\lesssim \|\nabla^{\le N} u'(0,\cd)\|_{L^2(\ext)} 
\\+
 \int_0^t\|\Gamma^{\le N} F(s,\cd)\|_{L^2(\ext)}\,ds
+\int_0^t \|[h^{IJ,\alpha\beta}(s,\cd)\partial_\alpha\partial_\beta,\Gamma^{\le
N}] u^J(s,\cd)\|_{L^2(\ext)}\,ds
\\+ \int_0^t\|\Gamma^{\le N-1} \Box
u(s,\cd)\|_{L^2(\ext)}ds
+ \int_0^t\Bigl\|\Bigl(|\partial
h(s,\cd)|+\frac{|h(s,\cd)|}{\la x\ra}\Bigr)|\Gamma^{\le N}\partial
u(s,\cd)|\Bigr\|_{L^2(\ext)}ds
\\+ \sum_{|\mu|+|\nu|\le N}\int_0^t 
\| |\Gamma^\nu h(s,\cd)| |\Gamma^\mu u'(s,\cd)|\|_{L^2((\ext)\cap\{|x|<1\})} \,ds
+ \|\Gamma^{\le N} G\|_{L^2L^2(S^\K_t)} \\+ \|\Gamma^{\le N-1} \Box
u\|_{L^2L^2(S^\K_t)} + \sup_{s\in [0,t]}\|\Gamma^{\le N-1} \Box u(s,\cd)\|_{L^2(\ext)}
\end{multline}
for any $N\ge 0$ and $t\ge 0$.
\end{proposition}

When $N=0$, \eqref{kss} is proved by pairing $\Box_h$ with
$\frac{r}{r+2^j}\partial_r u + \frac{3}{2}\frac{1}{r+2^j} u$ and
integrating by parts.  As $\partial_t$ perserves the boundary
conditions, the estimate when $\Gamma$ is replaced by $\partial_t$
follows immediately.  An elliptic regularity argument then yields the
estimate where the $\Gamma$ are all $\partial$.  Finally, the vector
fields can be broken into $\Omega_{ij} = (1-\beta(|x|))\Omega_{ij} +
\beta(|x|)\Omega_{ij}$ and $S=\Bigl(t\partial_t +
(1-\beta(|x|))r\partial_r\Bigr) + \beta(|x|)r\partial_r$.  The bounds
for any term involving $\beta(|x|)$ follow from the preceding step,
while the $(1-\beta(|x|))$ cutoffs guarantee preservation of the
boundary conditions.  Thus, the $N=0$ result may be applied and the
commutators are included in the $G$, which in turn is handled by the
estimate involving only derivatives as vector fields.

The result \eqref{kss} differs slightly in appearance from
\cite[Theorem 2.4]{ms_mathz}.  Here we have only applied the Schwarz
inequality and bootstrapped the portions coming from the multiplier.
As we are not using a multiple speed null condition, there is no loss
in the current setting in recording this simplied version.



\newsection{Sobolev-type decay bounds}
As was first noted in \cite{KSS}, local energy estimates allow us to
establish long time existence using decay in $|x|$ rather than decay
in $t$.  The latter, as in the Klainerman-Sobolev bounds
\cite{klainermanSob}, often relies on more symmetries of the equation.  This may
be manifest through the introduction of additional vector fields,
which may not mesh well with the existence of the boundary.

On the other hand, decay in $|x|$ can be readily established while only
relying on the generators of translations and spatial rotations, which
all have bounded coefficients on $\bdy$.  Indeed, we have the following.
\begin{lemma}[{\cite[Proposition 1]{klainermanSob}}]
  For $w\in C^\infty(\R^4)$ and $j\ge 0$,
  \begin{equation}
    \label{weighted-Sob}
    \|w\|_{L^\infty_{2^j}} \lesssim 2^{-3j/2}
    \|\Gamma^{\le 3} w\|_{L^2_{2^j}}.
  \end{equation}
\end{lemma}

To prove this, after localizing, one applies Sobolev embeddings on
$\R\times \S^3$.  The result follows upon adjusting the volume element
to match that of $\R^4$ in polar coordinates.  Due to tails from the
cutoff functions, the annulus on the right must be larger than that on
the left, but our notation allows for such.

One of the key new ideas in this article when compared with \cite{MS4}
is the following lemma from \cite[Lemma 3.11]{MTT}, which allows us
to obtain further decay when the solution is differentiated.  It will
be convenient to have a version of this estimate that holds
for small, time-dependent perturbations of $\Box$, which is what we provide.

\begin{lemma}  Assume that \eqref{speeds}
  holds.  Moreover, assume that \eqref{sym} and \eqref{smallpert} hold
  for $\delta>0$ sufficiently small.  Then,
for any $1\ll R \le c_1 T/8$, we have
\begin{multline}
  \label{mtt_lem}
  \|\partial w\|_{L^2L^2([T,2T]\times \{|x|\in [R,2R]\})} \lesssim
  R^{-1} \|S^{\le 1} w\|_{L^2L^2([T/2,4T]\times \{|x|\in [R/2,4R]\})} \\+ \|
  |\partial h| w\|_{L^2L^2([T/2,4T]\times\{|x|\in [R/2,4R]\})}+ R
  \|\Box_h w\|_{L^2L^2([T/2,4T]\times \{|x|\in [R/2,4R]\})}.
\end{multline}
\end{lemma}

\begin{proof}
  We begin with the pointwise estimate
\[|\nabla v|^2 \le \frac{1}{(ct-r)^2} (Sv)^2 + \frac{1}{c}\frac{t}{ct-r} [c^2|\nabla
v|^2 - (\partial_t v)^2],\quad r<c t.\]
We now fix $\chi\in C^\infty(\R\times \R)$ that is equal to $1$ on
$[1,2]\times [1,2]$ and supported in $[1/2,4]\times [1/2,4]$.
Applying the above estimate to each component of $w=(w^I)$, 
multiplying by $\chi_{T,R}(t,x) = \chi(t/c_1 T, |x|/R)$, and integrating
one obtains
\begin{multline*}
\int\int \chi_{T,R} |\nabla w|^2\,dx\,dt \lesssim \int\int
\frac{\chi_{T,R}}{T^2} |Sw|^2\,dx\,dt \\+ \sum_{I=1}^M\int \int
\chi_{T,R} \frac{t}{c_I(c_It-r)}[c_I^2 |\nabla w^I|^2 - (\partial_t w^I)^2]\,
dx\,dt.
\end{multline*}
As $\partial_t v = \frac{1}{t}Sv - \frac{r}{t}\partial_r v$, the above
generalizes to
\begin{multline*}
\int\int \chi_{T,R} |\partial w|^2\,dx\,dt \lesssim \int\int
\frac{\chi_{T,R}}{T^2} |Sw|^2\,dx\,dt \\+ \sum_{I=1}^M\int \int
\chi_{T,R} \frac{t}{c_I(c_It-r)} [c_I^2 |\nabla w^I|^2 - (\partial_t w^I)^2]\,
dx\,dt.
\end{multline*}
Integrating by parts, however, yields
\begin{multline*}\sum_{I=1}^M\int\int \chi_{T,R} \frac{t}{c_I(c_It-r)} 
[c_I^2 |\nabla w^I|^2 - (\partial_t w)^2] dx dt \\=
\sum_{I=1}^M\int\int \Box_h w^I \chi_{T,R} \frac{t}{c_I(c_It-r)} w^Idx
dt - \frac{1}{2}\sum_{I=1}^M\int\int \Box_{c_I}\Bigl( 
\chi_{T,R}\frac{t}{c_I(c_It-r)}\Bigr) (w^I)^2 dx dt
\\+\int \int \chi_{T,R}\frac{t}{c_I(c_It-r)} \Bigl[\partial_\alpha
h^{IJ,\alpha\beta}\partial_\beta w^J w^I
+ h^{IJ,\alpha\beta}\partial_\beta
w^J\partial_\alpha w^I\Bigr] dx dx\\+ \int \int \partial_\alpha\Bigl(
\chi_{T,R}\frac{t}{c_I(c_It-r)}\Bigr)  h^{IJ,\alpha\beta}\partial_\beta w^J
w^I dx dt.
\end{multline*}
The lemma follows from the two preceding equations as long as $\delta$
is sufficiently small so that the $h^{IJ,\alpha\beta} \partial_\beta
w^J \partial_\alpha w^I$ term may be bootstrapped.
\end{proof}

We end this section with a quick version of a Hardy inequality.  While
this is rather standard (see, e.g., \cite{Alinhac}), we provide a proof to illustrate that one may
rely only on the star-shapedness and not require the vanishing of the
solution on the boundary.  

\begin{lemma}\label{hardyLemma}
  Let $\K$ be a star-shaped with respect to the origin, and suppose
  $v\in C^1(\ext)$ vanishes at infinity. Then 
\[\|v/r\|_{L^2(\ext)}\lesssim
  \|\nabla v\|_{L^2(\ext)}.\]
\end{lemma}

\begin{proof}
  Integrating the identity
\[\partial_i \Bigl(\frac{x_i}{r^2} v^2\Bigr) = 2 \frac{v^2}{r^2} +
  2\frac{v}{r} \partial_r v\]
over $\ext$, we obtain
\[0\ge - \int_\bdy \la x, n\ra \frac{v^2}{r^2}\,d\sigma = 2
  \|v/r\|^2_{L^2(\ext)} + 2 \int \frac{v}{r} \partial_r v\,dx.\]
Here $n$ is the outward unit normal to $\K$.  An application of the
Schwarz inequality completes the proof.
\end{proof}

\newsection{Proof of Theorem \ref{main_thm}}
Since this is a small data result, higher order nonlinear terms are
better behaved.  Thus, for simplicity of exposition, we shall truncate
$Q$ to second order, which will allow us to assume 
\begin{equation}
  \label{Qtrunc}
  Q^I(u,u',u'') = a^\alpha_{IJK} u^J \partial_\alpha u^K +
  b^{\alpha\beta}_{IJK} \partial_\alpha u^J \partial_\beta u^K +
  A^{\alpha\beta}_{IJK} u^J \partial_\alpha\partial_\beta u^K +
  B^{\alpha\beta\gamma}_{IJK} \partial_\alpha
  u^J \partial_\beta\partial_\gamma u^K.
\end{equation}
The
constants $A^{\alpha\beta}_{IJK}, B^{\alpha\beta\gamma}_{IJK}$ are
subject to
\[A^{\alpha\beta}_{IJK} = A^{\beta\alpha}_{IJK}=A^{\alpha\beta}_{KJI},
  \quad B^{\alpha\beta\gamma}_{IJK} = B^{\alpha\gamma\beta}_{IJK} = B^{\alpha\beta\gamma}_{KJI}.\]
No $u^2$ term appears above
due to \eqref{restriction}. 

\subsubsection{Preliminaries}  Here we will prove a couple of lemmas
that will be useful for closing our iteration.  Recall that we have fixed $R_0>1$ so that $\K$ and the supports of the
data are contained in $\{|x|<R_0\}$.  Consider solutions $u$
to
\begin{equation}
  \label{genEq}
  \begin{cases}
    \Box_{c_I} u^I = a^\alpha_{IJK} \Bigl[\phi^J \partial_\alpha \tilde{u}^K +
\tilde{u}^J \partial_\alpha \psi_1^K\Bigr]
+b^{\alpha\beta}_{IJK}\Bigl[\partial_\alpha
\phi^J \partial_\beta\tilde{u}^K + \partial_\alpha
\tilde{u}^J \partial_\beta \psi_1^K\Bigr]
\\\qquad\qquad+
A^{\alpha\beta}_{IJK}\Bigl[\phi^J\partial_\alpha\partial_\beta u^K
+ \tilde{u}^J \partial_\alpha\partial_\beta \psi_2^K\Bigr]
+B^{\alpha\beta\gamma}_{IJK}\Bigl[\partial_\alpha\phi^J\partial_\beta\partial_\gamma
u^K + \partial_\alpha\tilde{u}^J \partial_\beta\partial_\gamma
\psi_2^K\Bigr],\\
u^I(t,\cd)|_{\bdy}=0,\\
u(0,\cd)=f,\quad \partial_t u(0,\cd) = g, \quad \supp (f,g)\subseteq
\{|x|\le R_0\}\cap (\ext).
  \end{cases}
\end{equation}
Here, $I=1,2,\dots, M$, $(t,x)\in \R_+\times\ext$, and $\phi, \psi_j,
\tilde{u} \in (C^\infty(\R_+\times \ext))^M$.  The initial data are moreover
assumed to be smooth and to satisfy the compatibility conditions to
infinite order. 

We shall set
\begin{multline}
  \label{MN}
  M_N[u](T) =   \sup_{t\in [0,T]} \Bigl[(1+t)^{-\delta} \|\Gamma^{\le N}
  u(t,\cd)\|_{L^2(\ext)} + \|\Gamma^{\le N}
  u'(t,\cd)\|_{L^2(\ext)}\Bigr] \\
+ \|\la x\ra^{-3/4} \Gamma^{\le N} u\|_{L^2L^2(S^\K_T)}+ \|\la x\ra^{-3/4}
  \Gamma^{\le N} u'\|_{L^2L^2(S^\K_T)}
\end{multline}
for some $\delta\ge 0$ sufficiently small.

 The first lemma controls the local energy portions of $M_N[u]$.
\begin{lemma}\label{lemmaMNle}
  For $u\in (C^\infty(\R_+\times \ext))^M$ solving \eqref{genEq} and any $N$, we have
\begin{multline}\label{MNstep1}\sup_{t\in [0,T]} \|\Gamma^{\le N}u'(t,\cd)\|_{L^2(\ext)}
+\|\la x\ra^{-3/4} \Gamma^{\le N} u'\|_{L^2L^2(S^\K_T)}
\\+\sup_{t\in [0,T]} \|\beta_{2R_0}(|x|) \Gamma^{\le N} u(t,\cd)\|_{L^2(\ext)}+
  \|\beta_{2R_0}(|x|) \Gamma^{\le N} u\|_{L^2L^2(S^\K_T)}  
\\\lesssim \|\nabla^{\le N} u'(0,\cd)\|_{L^2}+
M_{N/2+4}[\phi]\Bigl(M_N[\tilde{u}]+M_N[u]\Bigr) \\+ \Bigl(M_{N/2+3}[\tilde{u}]+M_{N/2+4}[u]\Bigr)
  M_N[\phi]
+ \Bigl(M_{N/2+3}[\psi_1] + M_{N/2+4}[\psi_2]\Bigr)M_N[\tilde{u}] 
\\+ M_{N/2+3}[\tilde{u}] \Bigl(M_N[\psi_1] + M_{N+1}[\psi_2]\Bigr)
\end{multline}
provided $|\partial^{\le 1} \phi| < \tilde{\delta}$ for some
$\tilde{\delta}>0$ sufficiently small.
\end{lemma}

The second lemma controls those portions of $M_N[u]$ that do not
contain a derivative and are away from the boundary.
\begin{lemma}\label{lemmaMNnod}
For $u\in (C^\infty(\R_+\times \ext))^M$ solving \eqref{genEq} and any $N$ and
any $t>0$, we have
\begin{multline}
  \label{MNstep2}
  (1+t)^{-\delta}\|(1-\beta_{R_0}(|x|)) \Gamma^{\le N}
  u(t,\cd)\|_{L^2(\ext)} + \|\la x\ra^{-3/4}
  (1-\beta_{R_0}(|x|))\Gamma^{\le N} u\|_{L^2L^2(S^\K_t)}\\
\lesssim \|\nabla^{\le N} u'(0,\cd)\|_{L^2}+ \Bigl(\Bigl[1+M_{N/2+2}[\phi]\Bigr]M_{N/2+3}[\tilde{u}]+M_{N/2+3}[u]\Bigr)
  M_N[\phi]
\\+
M_{N/2+4}[\phi]\Bigl(\Bigl[1+M_{N/2+4}[\phi]+M_{N/2+3}[\psi_1]+M_{N/2+4}[\psi_2]\Bigr]
M_N[\tilde{u}]+M_N[u]\Bigr) \\
+ \Bigl(M_{N/2+4}[\psi_1] + M_{N/2+4}[\psi_2]\Bigr)M_N[\tilde{u}] 
+ M_{N/2+3}[\tilde{u}] \Bigl[1+M_{N/2+2}[\phi]\Bigr]\Bigl(M_N[\psi_1] + M_{N+1}[\psi_2]\Bigr)
\\+ M_{N/2+2}[\phi] \|\la x\ra^{3/4} \Gamma^{\le N-1}\Box
\tilde{u}\|_{L^2L^2}
+M_N[\phi]\|\la x\ra^{3/4} \Gamma^{\le N/2+2}\Box
\tilde{u}\|_{L^2L^2}
\\+\sum_{j=1,2} M_{N/2+2}[\tilde{u}] \|\la x\ra^{3/4}\Gamma^{\le N}
\Box \psi_j\|_{L^2L^2} +\sum_{j=1,2} M_N[\tilde{u}] \|\la x\ra^{3/4} \Gamma^{\le
  N/2+3}\Box\psi_j\|_{L^2L^2}
\\+M_N[\tilde{u}] \|\la x\ra^{3/4}\Gamma^{\le N/2+2}\Box
\phi\|_{L^2L^2}
+M_N[u]\|\la x\ra^{3/4} \Gamma^{\le N/2+2}\Box \phi\|_{L^2L^2}
\\+M_N[\phi]\|\la x\ra^{3/4} \Gamma^{\le N/2+3}\Box u\|_{L^2L^2}
\end{multline}
provided $|\partial^{\le 1} \phi| < \tilde{\delta}$ for some
$\tilde{\delta}>0$ sufficiently small.\end{lemma}

We shall delay the proofs of the lemmas and will instead first
demonstrate that the lemmas can be used to complete the proof of
Theorem~\ref{main_thm}.  We will return to the proofs of the lemmas in
the last two subsections, which will then complete the proof of the
main result.

\subsubsection{Proof of Theorem~\ref{main_thm} assuming
  Lemma~\ref{lemmaMNle} and Lemma~\ref{lemmaMNnod} }
We shall solve \eqref{generalEquation} via iteration.  To that end, we
let $u_0\equiv 0$ and define $u_n$ for $n\ge 1$ to solve 
\begin{equation}
  \label{itEquation}
  \begin{cases}
    \Box_{c_I} u_n^I = Q^I(u_{n-1}, u_{n-1}', u_n''),\quad (t,x)\in \R_+\times
    \ext,\quad I=1,2,\dots,M\\
u^I_n(t,\cd)|_\bdy=0,\quad t\ge 0,\\
u^I_n(0,\cd) = f^I,\quad \partial_t u^I_n(0,\cd)=g^I.
  \end{cases}
\end{equation}
We shall begin by showing that the sequence $(M_N[u_n])_{n\in \N}$ is bounded
for $N$ sufficiently large, and we shall then subsequently use this to
show that the sequence $(u_n)$ is Cauchy and thus converges.

{\em Boundedness:} Our first goal is to argue inductively to show that there is a universal constant $C_0$ so
that 
\begin{equation}\label{bdd_goal}
M_N[u_n](T) \le 10 C_0\varepsilon
\end{equation}
for all $n\ge 0$ and all $T>0$ provided that $N$ is sufficiently large.

We first show that $M_N[u_1](T)\le C_0\varepsilon$.  Noticing that
$u_1$ solves \eqref{genEq} with $\phi, \psi_j, \tilde{u}\equiv 0$,
this base case follows immediately from \eqref{MNstep1}, \eqref{MNstep2},
and \eqref{smallness}.

We now assume that \eqref{bdd_goal} holds for $u_l$ for $l=1,2,\dots, n-1$ and use the
lemmas to establish the same for $u_n$.  To do so, we note that $u_n$
solves \eqref{genEq} with $\phi=\tilde{u}=u_{n-1}$ and $\psi_j\equiv
0$.  We also note that
\begin{multline*}|\Gamma^{\le N-1} \Box u_{n-1}|\lesssim |\Gamma^{\le
  (N-1)/2} \partial^{\le 1}u_{n-2}| |\Gamma^{\le N-1}\partial
u_{n-2}| 
\\+ |\Gamma^{\le (N-1)/2} \partial u_{n-2}| |\Gamma^{\le
  N-1}\partial^{\le 1} u_{n-2}|
+ |\Gamma^{\le (N-1)/2}\partial^{\le 1} u_{n-2}| |\Gamma^{\le
  N} \partial u_{n-1}| \\+ |\Gamma^{\le (N-1)/2 + 1}\partial u_{n-1}|
|\Gamma^{\le N-1}\partial^{\le 1} u_{n-2}|.
\end{multline*}
Thus, by \eqref{weighted-Sob}, we have
\begin{multline}\label{boxun1}
\|\la x\ra^{3/4} \Gamma^{\le N-1} \Box u_{n-1}\|_{L^2L^2} \lesssim
\|\la x\ra^{-3/4} \Gamma^{\le (N-1)/2 + 3} \partial^{\le 1}
u_{n-2}\|_{L^2L^2} \|\Gamma^{\le N-1} \partial u_{n-2}\|_{L^\infty
  L^2} 
\\+ \|\Gamma^{\le (N-1)/2+3} \partial u_{n-2}\|_{L^\infty L^2}  \|\la
x\ra^{-3/4} \Gamma^{\le N-1} \partial^{\le 1} u_{n-2}\|_{L^2L^2}
\\+\|\la x\ra^{-3/4} \Gamma^{\le (N-1)/2 + 3} \partial^{\le 1}
u_{n-2}\|_{L^2L^2} \|\Gamma^{\le N} \partial u_{n-1}\|_{L^\infty L^2} 
\\+ \|\Gamma^{\le (N-1)/2+4} \partial u_{n-1}\|_{L^\infty L^2} \|\la
x\ra^{-3/4} \Gamma^{\le N-1} \partial^{\le 1}u_{n-1}\|_{L^2L^2},
\end{multline}
which gives that this is 
\begin{multline*}\O\Bigl( M_{(N-1)/2+3}[u_{n-2}] M_{N-1}[u_{n-2}] +
M_{(N-1)/2+3}[u_{n-2}] M_N[u_{n-1}] \\+ M_{(N-1)/2+4}[u_{n-1}]
M_{N-1}[u_{n-2}]\Bigr).
\end{multline*}
Since $N\ge 12$, we have that $(N/2)+2 \le N-1$, so this same bound
may be applied for $\|\la x\ra^{3/4}\Gamma^{\le (N/2)+2} \Box
u_{n-1}\|_{L^2L^2}$.  Quite similarly, we have
\begin{multline}\label{boxun}
  \|\la x\ra^{3/4}\Gamma^{\le N-1}\Box u_{n}\|_{L^2L^2} = \O\Bigl( M_{(N-1)/2+3}[u_{n-1}] M_{N-1}[u_{n-1}] \\+
M_{(N-1)/2+3}[u_{n-1}] M_N[u_{n}] + M_{(N-1)/2+4}[u_{n}]
M_{N-1}[u_{n-1}]\Bigr).
\end{multline}

Assuming that $N\ge 12$ so that $(N/2)+4 \le N$, \eqref{MNstep1},
\eqref{MNstep2}, \eqref{smallness}, and the inductive hypothesis show that
\[  M_N[u_n] \le C_0 \varepsilon + C^2 (1+C\varepsilon) \varepsilon^2 +
  C \varepsilon (1+C\varepsilon) M_N[u_n],\]
which implies \eqref{bdd_goal} so long as $\varepsilon$ is
sufficiently small.

{\em Convergence:}  To complete the proof, we will show that
\begin{equation}
  \label{Cauchy_goal}
  M_{N-1}[u_n-u_{n-1}] (T) \le \Bigl(\frac{1}{2}\Bigr)^{n-1} M_{N-1}[u_1](T)
\end{equation}
for every $n$ and all $T>0$.  This will imply that the sequence is
Cauchy and, thus, convergent.  Per standard arguments, the limiting value is our desired
solution.

As the $n=1$ case is trivial, we shall first show the base case
\begin{equation}
  \label{Cauchy_basecase}
  M_{N-1}[u_2-u_1](T)\le \frac{1}{2} M_{N-1}[u_1](T).
\end{equation}
We shall then subsequently show that
\begin{equation}
  \label{Cauchy-goal-2}
  M_{N-1}[u_n - u_{n-1}](T) \le \frac{1}{4}M_{N-1}[u_{n-1}-u_{n-2}](T) +
  \frac{1}{8} M_{N-1}[u_{n-2}-u_{n-3}](T),\quad n\ge 3.
\end{equation}
Then \eqref{Cauchy_goal} follows from a straightforward argument
relying on strong induction, which completes the proof modulo the
proofs of Lemma \ref{lemmaMNle} and Lemma \ref{lemmaMNnod}.

\begin{proof}[Proof of \eqref{Cauchy_basecase}]
Observe that $u=u_2-u_1$ solves \eqref{genEq} with vanishing data and with
$\phi=\tilde{u}=\psi_2 = u_1$ and $\psi_1\equiv 0$.
Then \eqref{MNstep1}, \eqref{MNstep2}, and \eqref{bdd_goal} give
\begin{multline*}
  M_{N-1}[u_2-u_1] \le C\varepsilon \Bigl(M_{N-1}[u_1] +
  M_{N-1}[u_2-u_1]\Bigr) + C\varepsilon \|\la x\ra^{3/4} \Gamma^{\le
    N-1}\Box u_1\|_{L^2L^2} \\+ \|\la x\ra^{3/4} \Gamma^{\le N-1}\Box
  u_1\|_{L^2L^2} M_{N-1}[u_2-u_1] + M_{N-1}[u_1]\|\la x\ra^{3/4}\Gamma^{\le
    N-2} \Box (u_2-u_1)\|_{L^2L^2}
\end{multline*}
since $N\ge 12$ and $0<\varepsilon \ll 1$.  Note, however, that $\Box
u_1 = 0$.

We record that
\begin{multline*}
  |\Gamma^{\le N-2}\Box(u_2-u_1)| \lesssim |\Gamma^{\le N/2
    } \partial^{\le 1} u_1| |\Gamma^{\le N-1} \partial u_1| + 
  |\Gamma^{\le N/2}\partial u_1| | \Gamma^{\le N-2} u_1|
\\+ |\Gamma^{\le N/2-1} \partial^{\le 1} u_1| |\Gamma^{\le
  N-1}\partial(u_2-u_1)| + |\Gamma^{\le N/2} \partial(u_2-u_1)|
|\Gamma^{\le N-2} \partial^{\le 1} u_1|
\end{multline*}
since $N/2+4 \le N-2$ as $N\ge 12$.  Then, by \eqref{weighted-Sob}, we
see that
\begin{multline*}
  \|\la x\ra^{3/4} \Gamma^{\le N-2} \Box(u_2-u_1)\|_{L^2L^2} \lesssim
  \|\la x\ra^{-3/4} \Gamma^{\le N/2+3}\partial^{\le 1} u_1\|_{L^2L^2}
  \|\Gamma^{\le N-1} \partial u_1\|_{L^\infty L^2} \\+
  \|\Gamma^{\le N/2+3} \partial u_1\|_{L^\infty L^2} \|\la
  x\ra^{-3/4}\Gamma^{\le N-2} u_1\|_{L^2L^2} \\+ \|\la x\ra^{-3/4}
  \Gamma^{\le N/2+2}\partial^{\le 1} u_1\|_{L^2L^2} \|\Gamma^{\le
    N-1} \partial (u_2-u_1)\|_{L^\infty L^2} \\+ \|\Gamma^{\le
    N/2+3} \partial(u_2-u_1)\|_{L^\infty L^2} \|\la x\ra^{-3/4}
  \Gamma^{\le N-2}\partial^{\le 1} u_1\|_{L^2L^2}.
\end{multline*}
It then follows from \eqref{bdd_goal} that this is $\O(\varepsilon^2 +
\varepsilon M_{N-1}[u_2-u_1])$.  Plugging this into the above completes the proof
provided $\varepsilon$ is sufficiently small.
\end{proof}

\begin{proof}[Proof of \eqref{Cauchy-goal-2}]
  Similar to the above, we note that $u=u_n-u_{n-1}$ solves
  \eqref{genEq} with $f,g\equiv 0$, $\phi=u_{n-1}$,
$\tilde{u}=u_{n-1}-u_{n-2}$, $\psi_1=u_{n-2}$, and $\psi_2=u_{n-1}$.
Thus, \eqref{MNstep1}, \eqref{MNstep2}, and \eqref{bdd_goal} give 
\begin{multline}\label{Cauchy-goal-2a}
  M_{N-1}[u_n-u_{n-1}] \le C\varepsilon M_{N-1}[u_{n-1}-u_{n-2}]
  + C\varepsilon M_{N-1}[u_n-u_{n-1}]
\\ + C M_{N-1}[u_{n-1}-u_{n-2}] \sum_{j=1,2}
\|\la
x\ra^{3/4}\Gamma^{\le N-1} \Box u_{n-j}\|_{L^2L^2} \\+
M_{N-1}[u_n-u_{n-1}] \|\la x\ra^{3/4} \Gamma^{\le N-2}\Box
u_{n-1}\|_{L^2L^2} \\+C \varepsilon \|\la x\ra^{3/4}\Gamma^{\le N-2} \Box
(u_{n-1}-u_{n-2})\|_{L^2L^2}\\+ C\varepsilon \|\la x\ra^{3/4} \Gamma^{\le N-2}\Box (u_n-u_{n-1})\|_{L^2L^2}
\end{multline}
provided $N\ge 12$.  Notice that there is a loss of regularity
associated to the $\psi_2$ piece, which requires the boundedness in a
space with an extra degree of regularity.  This is characteristic of
quasilinear problems.

By \eqref{boxun1}, \eqref{boxun}, and \eqref{bdd_goal}, the third and fourth terms in the
right side of \eqref{Cauchy-goal-2} are controlled by the first two
terms.  It only remains to show
\begin{equation}\label{Cauchy-goal-2b}\|\la x\ra^{3/4} \Gamma^{\le
    N-2} \Box (u_j-u_{j-1})\|_{L^2L^2}\le C M_{N-1}[u_{j-1}-u_{j-2}]
+ C M_{N-1}[u_j-u_{j-1}]
\end{equation}
for $j=n-1$ or $j=n$.

To this end, we calculate
\begin{multline*}
  |\Gamma^{\le N-2} \Box (u_j-u_{j-1})| \lesssim |\Gamma^{\le
    N/2-1} \partial^{\le 1}
  u_{j-1}| |\Gamma^{\le N-2} \partial (u_{j-1}-u_{j-2})| \\+
  |\Gamma^{\le N/2-1}\partial(u_{j-1}-u_{j-2})| |\Gamma^{\le
    N-2} \partial^{\le 1} u_{j-1}|
+ |\Gamma^{\le N/2-1} \partial^{\le 1}
(u_{j-1}-u_{j-2})||\Gamma^{\le N-2} \partial u_{j-2}| 
\\+ |\Gamma^{\le N/2-1}\partial u_{j-2}| |\Gamma^{\le
  N-2} \partial^{\le 1}
(u_{j-1}-u_{j-2})|
\\+ |\Gamma^{N/2-1} \partial^{\le 1} u_{j-1}| |\Gamma^{\le
  N-1} \partial (u_j-u_{j-1})|
+ |\Gamma^{\le N/2} \partial (u_j-u_{j-1})| |\Gamma^{\le
  N-2}\partial^{\le 1} u_{j-1}|
\\+ |\Gamma^{\le N/2-1} \partial^{\le 1} (u_{j-1}-u_{j-2})| |\Gamma^{\le N-1} \partial
u_{j-1}| + |\Gamma^{\le N/2} \partial u_{j-1}| |\Gamma^{\le
  N-2} \partial^{\le 1}(u_{j-1}-u_{j-2})|.
\end{multline*}
Applying \eqref{weighted-Sob}, we obtain
\begin{multline*}
   \|\la x\ra^{3/4} \Gamma^{\le N-2} \Box (u_j-u_{j-1})\|_{L^2L^2}
   \\\lesssim \|\la x\ra^{-3/4} \Gamma^{\le
    N/2+2} \partial^{\le 1}
  u_{j-1}\|_{L^2L^2}  \|\Gamma^{\le N-2} \partial
  (u_{j-1}-u_{j-2})\|_{L^\infty L^2} \\+
 \|\Gamma^{\le N/2+2}\partial(u_{j-1}-u_{j-2})\|_{L^\infty L^2} \|\la
 x\ra^{-3/4} \Gamma^{\le
    N-2} \partial^{\le 1} u_{j-1}\|_{L^2L^2}
\\+ \|\la x\ra^{-3/4} \Gamma^{\le N/2+2} \partial^{\le 1}
(u_{j-1}-u_{j-2})\|_{L^2L^2} \|\Gamma^{\le N-2} \partial
u_{j-2}\|_{L^\infty L^2} 
\\+ \|\Gamma^{\le N/2+2}\partial u_{j-2}\|_{L^\infty L^2} \|\la
x\ra^{-3/4} \Gamma^{\le
  N-2} \partial^{\le 1}
(u_{j-1}-u_{j-2})\|_{L^2L^2}
\\+ \|\la x\ra^{-3/4} \Gamma^{N/2+2} \partial^{\le 1}
u_{j-1}\|_{L^2L^2}  \|\Gamma^{\le
  N-1} \partial (u_j-u_{j-1})\|_{L^\infty L^2}
\\+ \|\Gamma^{\le N/2+3} \partial (u_j-u_{j-1})\|_{L^\infty L^2} \|\la
x\ra^{-3/4} \Gamma^{\le
  N-2}\partial^{\le 1} u_{j-1}\|_{L^2L^2}
\\+ \|\la x\ra^{-3/4} \Gamma^{\le N/2+2} \partial^{\le 1} (u_{j-1}-u_{j-2})\|_{L^2L^2} \|\Gamma^{\le N-1} \partial
u_{j-1}\|_{L^\infty L^2} 
\\+ \|\Gamma^{\le N/2+3} \partial u_{j-1}\|_{L^\infty L^2} \|\la x\ra^{-3/4}\Gamma^{\le
  N-2} \partial^{\le 1}(u_{j-1}-u_{j-2})\|_{L^2L^2}.
\end{multline*}
From this, \eqref{Cauchy-goal-2b} follows immediately from \eqref{bdd_goal}.
\end{proof}

We now complete the proof of Theorem \ref{main_thm} by proving Lemma
\ref{lemmaMNle} and Lemma \ref{lemmaMNnod}.

\subsubsection{Proof of Lemma~\ref{lemmaMNle}}

We first note that the Dirichlet boundary conditions allow us to
  control the last two terms in the left side of \eqref{MNstep1} by the
  first two terms.  Here, the constant depends on $R_0$, but this is
  harmless.  Indeed, using the fact that $\partial_\alpha$,
  $\Omega_{ij}$, and $r\partial_r$ have bounded coefficients on the
  support of $\beta_{2R_0}$, we have
\begin{multline*}\|\Gamma^{\le N} u(t,\cd)\|_{L^2((\ext)\cap \{|x|<2R_0\})} \lesssim 
\|\Gamma^{\le N-1} u'(t,\cd)\|_{L^2((\ext)\cap \{|x|<2R_0\})}
\\+\|(t\partial_t)^{\le N} u(t,\cd)\|_{L^2((\ext)\cap \{|x|<2R_0\})}.
\end{multline*}
For the last term, we may use the fact that $t\partial_t$ preserves
the Dirichlet bounday conditions and integrate off of the boundary to
see that
\[\|(t\partial_t)^{\le N} u(t,\cd)\|_{L^2((\ext)\cap
    \{|x|<2R_0\})}\lesssim \|\Gamma^{\le N}
  u'(t,\cd)\|_{L^2((\ext)\cap \{|x|<2R_0\})}.\]

To bound the first two terms in \eqref{MNstep1}, we set
$h^{IK,\alpha\beta} = -A^{\alpha\beta}_{IJK}\phi^J -
B^{\gamma\alpha\beta}_{IJK} \partial_\gamma \phi^J$
and first note that
\begin{multline}\label{boxhvf}
  |\Gamma^{\le N} \Box_h u| +
  |[h^{IK,\alpha\beta}\partial_\alpha\partial_\beta, \Gamma^{\le N}]
  u^K| + |\Gamma^{\le N-1} \Box u|
\\\lesssim
|\Gamma^{\le N/2+1} \phi | |\Gamma^{\le N} \partial
  \tilde{u}| 
+|\Gamma^{\le N/2}\partial\tilde{u}| |\Gamma^{\le N} \partial^{\le 1}
\phi|
\\+
|\Gamma^{\le N/2} \partial \psi_1 | |\Gamma^{\le N} \partial^{\le 1}
  \tilde{u}| 
+|\Gamma^{\le N/2}\partial^{\le 1}\tilde{u}| |\Gamma^{\le N} \partial
\psi_1|
\\+|\Gamma^{\le N/2+1} \phi| |\Gamma^{\le N}\partial u| + |\Gamma^{\le
  N/2+1}\partial u| |\Gamma^{\le N} \partial^{\le 1} \phi|
\\+|\Gamma^{\le N/2}\partial^{\le 1} \tilde{u}| |\Gamma^{\le N+1}\partial \psi_2|
+ |\Gamma^{N/2+1} \partial\psi_2| |\Gamma^{\le N} \partial^{\le 1} \tilde{u}|.
\end{multline}
Due to \eqref{kss}, it suffices to control each of these terms in
$L^1L^2$, $L^2L^2$, and $L^\infty L^2$.

To control the terms in $L^1 L^2$, we argue as in \cite{KSS} and note that
\eqref{weighted-Sob} and the Cauchy-Schwarz inequality give
\begin{align*}
\int_0^T \|v w\|_{L^2}\,ds &\le \sum_{j\ge 0} \int_0^T \|v
w\|_{L^2_{2^j}}\, ds \lesssim \sum_{j\ge 0} \int_0^T 2^{-3j/2}
\|\Gamma^{\le 3} v\|_{L^2_{2^j}} \|w\|_{L^2_{2^j}}\,ds \\&\lesssim \sum_{j\ge
  0} \|\la x\ra^{-3/4} \Gamma^{\le 3} v\|_{L^2 L^2_{2^j}} \|\la
x\ra^{-3/4} w\|_{L^2L^2_{2^j}}\\
&\lesssim \|\la x\ra^{-3/4} \Gamma^{\le 3} v\|_{L^2L^2} \|\la x\ra^{-3/4} w\|_{L^2L^2}.
\end{align*}
Applying this to each term in \eqref{boxhvf} where $v$ is chosen to
be the lower order term, we see that the $L^1L^2$-norm of
\eqref{boxhvf} is controlled by a constant times
\begin{multline}\label{L1L2}
M_{N/2+4}[\phi]\Bigl(M_N[\tilde{u}]+M_N[u]\Bigr) + \Bigl(M_{N/2+3}[\tilde{u}]+M_{N/2+4}[u]\Bigr)
  M_N[\phi]\\
+ \Bigl(M_{N/2+3}[\psi_1] + M_{N/2+4}[\psi_2]\Bigr)M_N[\tilde{u}] 
+ M_{N/2+3}[\tilde{u}] \Bigl(M_N[\psi_1] + M_{N+1}[\psi_2]\Bigr).
\end{multline}

We argue similarly to control the terms in $L^2L^2$.  Here
\eqref{weighted-Sob} gives an excess of decay.  We keep the $L^2_t$
norm on the factor that may not contain a derivative so there is no
loss of decay as is associated with the first term in \eqref{MN}.
When each term in the right side of \eqref{boxhvf} is measured in
$L^2L^2$, we see that it is
\begin{multline*}
  \lesssim \|\la x\ra^{-3/4} \Gamma^{\le N/2+4} \phi \|_{L^2L^2} \Bigl(\|\Gamma^{\le N} \partial
  \tilde{u}\|_{L^\infty L^2} + \|\Gamma^{\le N} \partial u\|_{L^\infty L^2}\Bigr)
\\+\Bigl(\|\Gamma^{\le N/2+3}\partial\tilde{u}\|_{L^\infty L^2} +
\|\Gamma^{\le N/2 +4} \partial u\|_{L^\infty L^2}\Bigr) \|\la x\ra^{-3/4}\Gamma^{\le N} \partial^{\le 1}
\phi\|_{L^2L^2}
\\+
\Bigl(\|\Gamma^{\le N/2+3} \partial \psi_1\|_{L^\infty L^2} +
\|\Gamma^{N/2+4}\partial \psi_2\|_{L^\infty L^2}
\Bigr)\|\la x\ra^{-3/4}\Gamma^{\le N} \partial^{\le 1}
  \tilde{u}\|_{L^2L^2} 
\\+\|\la x\ra^{-3/4} \Gamma^{\le N/2+3}\partial^{\le 1}\tilde{u}\|_{L^2L^2}\Bigl(\|\Gamma^{\le N} \partial
\psi_1\|_{L^\infty L^2} + \|\Gamma^{\le N+1}\partial\psi_2\|_{L^\infty L^2}\Bigr),
\end{multline*}
which is in turn controlled by \eqref{L1L2}.

Finally, when measuring each term in $L^\infty L^2$, we apply
\eqref{weighted-Sob} and pair a portion of the decay with the term
that may not include a derivative as this allows for an application of
a Hardy inequality to prevent having to handle the growth in $t$ that
is coupled to the first term in \eqref{MN}.  Thus, when the right side
of \eqref{boxhvf} is in $L^\infty L^2$, we have that it is
\begin{multline*}
  \lesssim \|\la x\ra^{-1} \Gamma^{\le N/2+4} \phi \|_{L^\infty L^2} \Bigl(\|\Gamma^{\le N} \partial
  \tilde{u}\|_{L^\infty L^2} + \|\Gamma^{\le N} \partial u\|_{L^\infty L^2}\Bigr)
\\+\Bigl(\|\Gamma^{\le N/2+3}\partial\tilde{u}\|_{L^\infty L^2} +
\|\Gamma^{\le N/2 +4} \partial u\|_{L^\infty L^2}\Bigr) \|\la x\ra^{-1}\Gamma^{\le N} \partial^{\le 1}
\phi\|_{L^\infty L^2}
\\+
\Bigl(\|\Gamma^{\le N/2+3} \partial \psi_1\|_{L^\infty L^2} +
\|\Gamma^{N/2+4}\partial \psi_2\|_{L^\infty L^2}
\Bigr)\|\la x\ra^{-1}\Gamma^{\le N} \partial^{\le 1}
  \tilde{u}\|_{L^\infty L^2} 
\\+\|\la x\ra^{-1} \Gamma^{\le N/2+3}\partial^{\le
  1}\tilde{u}\|_{L^\infty L^2}\Bigl(\|\Gamma^{\le N} \partial
\psi_1\|_{L^\infty L^2} + \|\Gamma^{\le N+1}\partial\psi_2\|_{L^\infty L^2}\Bigr).
\end{multline*}
An application of Lemma \ref{hardyLemma} when the weighted terms do
not include $\partial$ shows that these terms are also bounded by
\eqref{L1L2}, which completes the proof.\qed

\subsubsection{Proof of Lemma~\ref{lemmaMNnod}}
For $|\mu|\le N$, we note that $(1-\beta_{R_0}(|x|))\Gamma^\mu u^I$
solves the boundaryless equation
\begin{multline*}\Box_{c_I} (1-\beta_{R_0}(|x|))\Gamma^\mu u^I = c_I^2 [\Delta,
  \beta_{R_0}(|x|)] \Gamma^\mu u^I
+ (1-\beta_{R_0}(|x|)) [\Box_{c_I}, \Gamma^\mu] u^I \\+
(1-\beta_{R_0}(|x|)) \Gamma^\mu \Box_{c_I}u^I
\end{multline*}
with vanishing initial data due to the support conditions on $f, g$.
By using \eqref{mtt_lem}, extra decay is available for terms
involving derivatives, but it comes at the cost of a vector field.
Thus, care is required when the highest number of vector fields lands
on the term with derivatives.  Similar care is required to prevent
having a term with too many vector fields resulting from the full
number of vector fields landing on a factor with two derivatives.  To
handle this, such terms are broken into a piece where the nonlinearity
is in divergence form, for which we have \eqref{divFormLE}, and a piece
where the trouble is avoided.

To this end, for any $|\mu|\le N$, we write
\begin{align*}
 (1-&\beta_{R_0}(|x|))\Gamma^\mu (a^\alpha_{IJK} \phi^J \partial_\alpha
 \tilde{u}^K) \\&=  \Bigl[(1-\beta_{R_0}(|x|))\Gamma^\mu (a^\alpha_{IJK} \phi^J \partial_\alpha \tilde{u}^K) - (1-\beta_{R_0}(|x|)) a^\alpha_{IJK}
\phi^J \partial_\alpha (\Gamma^\mu \tilde{u}^K)\Bigr] \\&\qquad\qquad-
(1-\beta_{R_0}(|x|)) a^\alpha_{IJK}\partial_\alpha \phi^J \Gamma^\mu \tilde{u}^K
+\partial_\alpha \Bigl[(1-\beta_{R_0}(|x|))
a^\alpha_{IJK} \phi^J \Gamma^\mu \tilde{u}^K\Bigr] 
\\&\qquad\qquad\qquad\qquad+\beta_{R_0}'(|x|) \frac{x_k}{r}a^k_{IJK} \phi^J
\Gamma^\mu \tilde{u}^K\\
&=\O\Bigl(|\Gamma^{\le N/2} \phi| |\Gamma^{\le N-1} \partial \tilde{u}| + |\Gamma^{\le
  N/2} \partial \tilde{u}| |\Gamma^{\le N} \phi| + |\partial
  \phi| |\Gamma^{\le N} \tilde{u}|\Bigr)
\\&\qquad\qquad+\partial_\alpha \Bigl[(1-\beta_{R_0}(|x|))
a^\alpha_{IJK} \phi^J \Gamma^\mu \tilde{u}^K\Bigr] 
+\beta_{R_0}'(|x|) \frac{x_k}{r}a^k_{IJK} \phi^J
\Gamma^\mu \tilde{u}^K.
\end{align*}
The terms involving $\psi_j$ allow for a loss of regularity.  As such,
less care is required here, and we simply note that
\[(1-\beta_{R_0}(|x|)) \Gamma^\mu (a^\alpha_{IJ}
  \tilde{u}^J \partial_\alpha \psi_1^K)=\O\Bigl(|\Gamma^{\le N/2}
  \tilde{u}| |\Gamma^{\le N} \partial \psi_1| + |\Gamma^{\le
    N/2}\partial \psi_1| |\Gamma^{\le N} \tilde{u}|\Bigr).
\]
As 
\begin{multline*}(1-\beta_{R_0}(|x|))\Gamma^\mu
  (b^{\alpha\beta}_{IJK} \partial_\alpha \phi^J \partial_\beta
  \tilde{u}^K + \partial_\alpha \tilde{u}^J \partial_\beta \psi_1^K) \\=
  \O\Bigl(\Bigl[|\Gamma^{\le N/2} \partial \phi| + |\Gamma^{\le
    N/2} \partial \psi_1|\Bigr] |\Gamma^{\le N}\partial
  \tilde{u}|
+ |\Gamma^{\le N/2} \partial \tilde{u}|\Bigl[|\Gamma^{\le N} \partial
\phi|+|\Gamma^{\le N} \partial \psi_1|\Bigr] 
\Bigr),
\end{multline*}
no modification is needed for this term.  
However, for the remaining
pieces, we similarly arrange as follows:
\begin{align*}
 (1&-\beta_{R_0}(|x|))\Gamma^\mu(A^{\alpha\beta}_{IJK}\phi^J \partial_\alpha\partial_\beta
u^K) \\&= \Bigl[(1-\beta_{R_0}(|x|))\Gamma^\mu(A^{\alpha\beta}_{IJK}\phi^J \partial_\alpha\partial_\beta
u^K) - (1-\beta_{R_0}(|x|)) A^{\alpha\beta}_{IJK} \phi^J \partial_\alpha \partial_\beta
\Gamma^\mu u^K\Bigr] \\&\qquad\qquad-
(1-\beta_{R_0}(|x|)) A^{\alpha\beta}_{IJK} \partial_\alpha \phi^J \partial_\beta \Gamma^\mu
u^K + \partial_\alpha \Bigl[(1-\beta_{R_0}(|x|))A^{\alpha\beta}_{IJK}
\phi^J \partial_\beta \Gamma^\mu u^K\Bigr]
\\&\qquad\qquad\qquad+ \beta'_{R_0}(|x|)\frac{x_k}{r}
A^{k\beta}_{IJK}\phi^J\partial_\beta\Gamma^\mu u^K\\
&=\O\Bigl(|\Gamma^{\le N/2} \phi||\Gamma^{\le N-1} \partial^2 u| +
  |\Gamma^{\le N/2+1} \partial u||\Gamma^{\le N} \phi|+ |\partial
 \phi| |\Gamma^{\le N} \partial u|\Bigr)
\\&\qquad\qquad\qquad+ \partial_\alpha \Bigl[(1-\beta_{R_0}(|x|))A^{\alpha\beta}_{IJK}
\phi^J \partial_\beta \Gamma^\mu u^K\Bigr]+
\beta'_{R_0}(|x|)\frac{x_k}{r}
A^{k\beta}_{IJK}\phi^J\partial_\beta\Gamma^\mu u^K
\end{align*}
and
\begin{align*}
 (1&-\beta_{R_0}(|x|))\Gamma^\mu(B^{\alpha\beta\gamma}_{IJK}\partial_\gamma \phi^J \partial_\alpha\partial_\beta
u^K) \\&=\Bigl[
           (1-\beta_{R_0}(|x|))\Gamma^\mu(B^{\alpha\beta\gamma}_{IJK}\partial_\gamma
           \phi^J \partial_\alpha\partial_\beta 
u^K) - (1-\beta_{R_0}(|x|)) B^{\alpha\beta\gamma}_{IJK} \partial_\gamma \phi^J \partial_\alpha\partial_\beta
\Gamma^\mu u^K \Bigr]\\&\qquad-(1-\beta_{R_0}(|x|))
B^{\alpha\beta\gamma}_{IJK} \partial_\alpha \partial_\gamma \phi^J \partial_\beta \Gamma^\mu
u^K + \partial_\alpha \Bigl[(1-\beta_{R_0}(|x|)) B^{\alpha\beta\gamma}_{IJK}
\partial_\gamma \phi^J \partial_\beta \Gamma^\mu u^K\Bigr]
\\&\qquad\qquad\qquad+\beta'_{R_0}(|x|)\frac{x_k}{r}B^{k\beta\gamma}_{IJK}\partial_\gamma
\phi^J\partial_\beta \Gamma^\mu u^K
\\&=\O\Bigl(|\Gamma^{\le N/2} \partial \phi||\Gamma^{\le
    N}\partial u| + |\partial^{\le N/2+1} \partial u||\Gamma^{\le
    N} \partial \phi |\Bigr)
\\&\qquad\qquad+ \partial_\alpha \Bigl[(1-\beta_{R_0}(|x|)) B^{\alpha\beta\gamma}_{IJK}
\partial_\gamma \phi^J \partial_\beta \Gamma^\mu u^K\Bigr]
+\beta'_{R_0}(|x|)\frac{x_k}{r}B^{k\beta\gamma}_{IJK}\partial_\gamma
\phi^J\partial_\beta \Gamma^\mu u^K.
\end{align*}
We similarly have these last two rearrangements when $(\phi, u)$ is
replaced by $(\tilde{u}, \psi_2)$.

We write $(1-\beta_{R_0}(|x|))\Gamma^\mu u^I = \tilde{w} + w + v$
where each has vanishing Cauchy data and
\begin{multline*}\Box_{c_I} \tilde{w} = c_I^2 [\Delta,\beta_{R_0}(|x|)] \Gamma^\mu
  u^I +\beta_{R_0}'(|x|) \frac{x_k}{r}a^k_{IJK} \phi^J
\Gamma^\mu \tilde{u}^K \\+
\beta'_{R_0}(|x|)\frac{x_k}{r}
A^{k\beta}_{IJK}\phi^J \partial_\beta\Gamma^\mu u^K
+\beta'_{R_0}(|x|)\frac{x_k}{r}B^{k\beta\gamma}_{IJK}\partial_\gamma 
\phi^J\partial_\beta \Gamma^\mu u^K 
\\+
\beta'_{R_0}(|x|)\frac{x_k}{r}
A^{k\beta}_{IJK}\tilde{u}^J \partial_\beta\Gamma^\mu \psi_2^K
+\beta'_{R_0}(|x|)\frac{x_k}{r}B^{k\beta\gamma}_{IJK}\partial_\gamma 
\tilde{u}^J\partial_\beta \Gamma^\mu \psi_2^K 
\end{multline*}
and
\begin{multline*}\Box_{c_I} v = \partial_\alpha \Bigl[(1-\beta_{R_0}(|x|))a^\alpha_{IJK}
  \phi^J \Gamma^\mu \tilde{u}^K\Bigr] +\partial_\alpha
  \Bigl[(1-\beta_{R_0}(|x|))A^{\alpha\beta}_{IJK} 
\phi^J \partial_\beta \Gamma^\mu u^K\Bigr] \\+ 
\partial_\alpha \Bigl[(1-\beta_{R_0}(|x|))B^{\alpha\beta\gamma}_{IJK}
\partial_\gamma \phi^J \partial_\beta \Gamma^\mu u^K\Bigr]
+\partial_\alpha
  \Bigl[(1-\beta_{R_0}(|x|))A^{\alpha\beta}_{IJK} 
\tilde{u}^J \partial_\beta \Gamma^\mu \psi_2^K\Bigr] \\+ 
\partial_\alpha \Bigl[(1-\beta_{R_0}(|x|))B^{\alpha\beta\gamma}_{IJK}
\partial_\gamma \tilde{u}^J \partial_\beta \Gamma^\mu \psi_2^K\Bigr].
\end{multline*}

We shall apply \eqref{cpctF} to $\tilde{w}$, \eqref{divFormLE} to $v$, and
\eqref{lenod}, \eqref{weightedStrichartz} to $w$.  Upon doing so, it
follows that the left side of \eqref{MNstep2} is controlled by

\begin{multline}
  \label{lotStep1}
\| [\Delta,\beta_{R_0}(|x|)] \Gamma^{\le N}
  u\|_{L^2L^2} +
\|\beta'_{R_0} \phi\:
\Gamma^{\le N} \tilde{u}\|_{L^2L^2} \\+ \|\beta'_{R_0}
\partial^{\le 1} \phi \:\Gamma^{\le N} \partial
u\|_{L^2L^2} 
+ \|\beta'_{R_0}
\partial^{\le 1} \tilde{u} \:\Gamma^{\le N} \partial
\psi_2\|_{L^2L^2} 
\\+ \int_0^t \| \phi\:
\Gamma^{\le N} \tilde{u} \|_{L^2(\ext)}\,ds + \int_0^t \|\partial^{\le
  1} \phi\: \Gamma^{\le N} \partial u\|_{L^2(\ext)}\,ds 
+\int_0^t \|\partial^{\le
  1} \tilde{u}\: \Gamma^{\le N} \partial \psi_2\|_{L^2(\ext)}\,ds
\\
+
\la t\ra^{-\delta} \|r^{-1} \Gamma^{\le N/2} \phi \:\Gamma^{\le
  N-1} \partial \tilde{u} \|_{L^1L^1L^2} 
+
\|r^{-\frac{5}{4}} \Gamma^{\le N/2} \phi \:\Gamma^{\le
  N-1} \partial \tilde{u} \|_{L^1L^1L^2}
\\+
\la t\ra^{-\delta} \|r^{-1} \Gamma^{\le N/2} \partial \phi \:\Gamma^{\le
  N} \partial^{\le 1} \tilde{u} \|_{L^1L^1L^2} 
+
\|r^{-\frac{5}{4}} \Gamma^{\le N/2} \partial \phi \:\Gamma^{\le
  N} \partial^{\le 1} \tilde{u} \|_{L^1L^1L^2}
\\+ \la t\ra^{-\delta} \|r^{-1} \Gamma^{\le
  N/2} \partial \tilde{u}\:\Gamma^{\le N} \partial^{\le 1}\phi \|_{L^1L^1L^2} 
+
\|r^{-\frac{5}{4}} \Gamma^{\le
  N/2} \partial \tilde{u} \:\Gamma^{\le N} \partial^{\le 1}\phi \|_{L^1L^1L^2}
\\
+\sum_{j=1,2}\Bigl(
\la t\ra^{-\delta} \|r^{-1} \Gamma^{\le N/2} \partial^{\le 1}\tilde{u} \:\Gamma^{\le
  N} \partial \psi_j\|_{L^1L^1L^2} 
+
\|r^{-\frac{5}{4}} \Gamma^{\le N/2} \partial^{\le 1}\tilde{u} \:\Gamma^{\le
  N} \partial \psi_j \|_{L^1L^1L^2}\Bigr)
\\+ \sum_{j=1,2}\Bigl(\la t\ra^{-\delta} \|r^{-1} \Gamma^{\le
  N/2+1} \partial \psi_j \:\Gamma^{\le N} \partial^{\le 1}\tilde{u}\|_{L^1L^1L^2} 
+
\|r^{-\frac{5}{4}} \Gamma^{\le
  N/2+1} \partial \psi_j \:\Gamma^{\le N} \partial^{\le 1}\tilde{u}
\|_{L^1L^1L^2}\Bigr)
\\
+
\la t\ra^{-\delta} \|r^{-1} \Gamma^{\le N/2} \phi \:\Gamma^{\le
  N-1} \partial^2 u \|_{L^1L^1L^2} 
+
\|r^{-\frac{5}{4}} \Gamma^{\le N/2} \phi \:\Gamma^{\le
  N-1} \partial^2 u \|_{L^1L^1L^2}
\\+
\la t\ra^{-\delta} \|r^{-1} \Gamma^{\le N/2} \partial \phi \:\Gamma^{\le
  N} \partial u \|_{L^1L^1L^2} 
+
\|r^{-\frac{5}{4}} \Gamma^{\le N/2} \partial \phi \:\Gamma^{\le
  N} \partial u \|_{L^1L^1L^2}
\\+ \la t\ra^{-\delta} \|r^{-1} \Gamma^{\le
  N/2+1} \partial u\:\Gamma^{\le N} \partial^{\le 1}\phi \|_{L^1L^1L^2} 
+
\|r^{-\frac{5}{4}} \Gamma^{\le
  N/2+1} \partial u \:\Gamma^{\le N} \partial^{\le 1} \phi \|_{L^1L^1L^2}.
\end{multline}

By \eqref{MNstep1}, the first term is controlled by the right side of
\eqref{MNstep2}.   By a Sobolev embedding, the second, third, and fourth term are
controlled by
\begin{multline*}\|\beta_{2R_0} \partial^{\le 3}\partial 
  \phi\|_{L^\infty L^2} \Bigl(\|\la
x\ra^{-3/4} \Gamma^{\le N} \tilde{u} \|_{L^2L^2} + \|\la
x\ra^{-3/4} \Gamma^{\le N} \partial u\|_{L^2L^2}\Bigr)
\\+ \|\beta_{2R_0} \partial^{\le 3}\partial \tilde{u}\|_{L^\infty L^2}
\|\la x\ra^{-3/4} \Gamma^{\le N} \partial\psi_2\|_{L^2L^2}
\\\lesssim M_3[\phi]\Bigl(M_N[\tilde{u}] +M_N[u]\Bigr) + M_3[\tilde{u}]M_N[\psi_2].
\end{multline*}
The fifth and sixth terms are bounded by applying \eqref{weighted-Sob} to the lower order
factor and applying the Cauchy-Schwarz inequality.  This gives that
these terms are controlled by
\begin{multline*}\|\la x\ra^{-3/4} \Gamma^{\le 3} \partial^{\le 1} \phi\|_{L^2L^2} \Bigl(\|\la
x\ra^{-3/4} \Gamma^{\le N} \tilde{u}\|_{L^2L^2} + \|\la x\ra^{-3/4}
\Gamma^{\le N} \partial u\|_{L^2L^2}\Bigr) \\\lesssim
M_3[\phi]\Bigl(M_N[\tilde{u}]+M_N[u]\Bigr).
\end{multline*}
The same argument shows that the seventh term in \eqref{lotStep1} is
$\O(M_3[\tilde{u}] M_N[\psi_2])$.

To control the remaining terms, except those that involve
$\partial^2 u$, we shall show
\begin{multline}
  \label{abstractGoal}
  \la t\ra^{-\delta} \|r^{-1} \partial^{\le 1} \Phi \: \partial \Psi \|_{L^1L^1L^2} 
+
\|r^{-\frac{5}{4}} \partial^{\le 1} \Phi \: \partial \Psi\|_{L^1L^1L^2} \lesssim
M_2[\Phi](t)M_1[\Psi](t) \\+ M_2[\Phi](t)\|\la x\ra^{3/4} \Box \Psi\|_{L^2L^2}
\end{multline}
and
\begin{multline}
  \label{abstractGoal2}
  \la t\ra^{-\delta} \|r^{-1} \partial^{\le 1} \Phi
  \: \partial \Psi \|_{L^1L^1L^2} 
+
\|r^{-\frac{5}{4}} \partial^{\le 1} \Phi \: \partial \Psi\|_{L^1L^1L^2} \lesssim
M_0[\Phi](t)M_3[\Psi](t) \\+ M_0[\Phi](t)\|\la x\ra^{3/4} \Gamma^{\le 2} \Box \Psi\|_{L^2L^2}.
\end{multline}

\begin{proof}[Proof of \eqref{abstractGoal} and \eqref{abstractGoal2}]
We divide the analysis
into the region where $|x|<c_1 s/8$ and $|x|\ge c_1s/8$.  Applying Sobolev embeddings and the Schwarz
inequality on $\S^3$, these terms are bounded by 
\begin{multline}\label{dividedAnalysis}
\la t\ra^{-\delta} \|r^{-1} \partial^{\le 1}\Phi \: \partial \Psi \|_{L^1L^1L^2} 
+
\|r^{-\frac{5}{4}} \partial^{\le 1} \Phi \: \partial \Psi \|_{L^1L^1L^2}
\\\lesssim
  \sum_{j<\log t} \sum_{k<j-\tilde{C}} 2^{-k} \|\Gamma^{\le 2} \partial^{\le 1}
  \Phi\|_{L^2_{2^j}
    L^2_{2^k}} \|\partial \Psi\|_{L^2_{2^j}L^2_{2^k}}
\\+ \la t\ra^{-\delta} \int_0^t \int_{r\ge c_1s/4} \la r\ra^{-1} \|\Gamma^{\le
  2} \partial^{\le 1} \Phi(t,r\cdot)\|_{L^2(\S^3)} \|\partial \Psi(t,r\cdot)\|_{L^2(\S^3)} r^2 dr ds
\\+ \int_0^t \int_{r\ge c_1s/4} \la r\ra^{-5/4} \|\Gamma^{\le
  2} \partial^{\le 1} \Phi(t,r\cdot)\|_{L^2(\S^3)} \|\partial \Psi(t,r\cdot)\|_{L^2(\S^3)} r^2 dr ds
\end{multline}
for some fixed $\tilde{C}>0$.  Here we have applied the Sobolev
embedding in the fashion that will yield \eqref{abstractGoal}.  To prove
\eqref{abstractGoal2}, we instead apply the Sobolev estimate to $\Psi$
and proceed with the same steps.

Applying the Schwarz inequality, the second and third terms in the
right side of \eqref{dividedAnalysis} are controlled by
\[  \Bigl(\la t\ra^{-\delta} \int_0^t \la s\ra^{-1+\delta} \,ds +
  \int_0^t \la s\ra^{-5/4+\delta}\,ds\Bigr)
M_2[\Phi](t)M_0[\Psi](t)  
\lesssim M_2[\Phi](t) M_0[\Psi](t).
\]

To bound the first term in the right side of \eqref{dividedAnalysis},
we note that \eqref{mtt_lem} (with $h\equiv 0$) gives
\[\|\partial \Psi\|_{L^2_{2^j}L^2_{2^k}} \lesssim 2^{-k} \|\Gamma^{\le 1}
  \Psi\|_{L^2_{2^j}L^2_{2^k}} + 2^k \|\Box \Psi\|_{L^2_{2^j}L^2_{2^k}}\]
when $k<j-\tilde{C}$.  Thus,
\begin{align*}
    \sum_{j<\log t} &\sum_{k<j-\tilde{C}} 2^{-k} \|\Gamma^{\le 2} \partial^{\le
             1}\Phi\|_{L^2_{2^j}
    L^2_{2^k}} \|\partial \Psi\|_{L^2_{2^j}L^2_{2^k}}
\\&\lesssim \sum_j \sum_{k<j-\tilde{C}} \Bigl(2^{-2k} \|\Gamma^{\le
    2} \partial^{\le 1} \Phi\|_{L^2_{2^j}
    L^2_{2^k}} \|\Gamma^{\le 1} \Psi\|_{L^2_{2^j}L^2_{2^k}} +
  \|\Gamma^{\le 2} \partial^{\le 1} \Phi\|_{L^2_{2^j}L^2_{2^k}} \|\Box
    \Psi\|_{L^2_{2^j}L^2_{2^k}}\Bigr)
\\&\lesssim M_2[\Phi](t)M_1[\Psi](t)+M_2[\Phi](t) \|\la x\ra^{3/4} \Box \Psi\|_{L^2L^2},
\end{align*}
where we have applied the Cauchy-Schwarz inequality to sum over $k, j$.
\end{proof}

Examining the remaining terms of \eqref{lotStep1}, we see that
\eqref{abstractGoal} gives us
\begin{multline*}
\la t\ra^{-\delta} \|r^{-1} \Gamma^{\le N/2} \phi \:\Gamma^{\le
  N-1} \partial \tilde{u} \|_{L^1L^1L^2} 
+
\|r^{-\frac{5}{4}} \Gamma^{\le N/2} \phi \:\Gamma^{\le
  N-1} \partial \tilde{u} \|_{L^1L^1L^2}\\\lesssim M_{N/2+2}[\phi]
M_N[\tilde{u}]
+M_{N/2+2}[\phi] \|\la x\ra^{3/4} \Gamma^{\le N-1} \Box \tilde{u}\|_{L^2L^2}
\end{multline*}
and
\begin{multline*}
  \la t\ra^{-\delta} \|r^{-1} \Gamma^{\le N/2} \partial^{\le 1}\tilde{u} \:\Gamma^{\le
  N} \partial \psi_j\|_{L^1L^1L^2} 
+
\|r^{-\frac{5}{4}} \Gamma^{\le N/2} \partial^{\le 1}\tilde{u} \:\Gamma^{\le
  N} \partial \psi_j \|_{L^1L^1L^2}
\\\lesssim M_{N/2+2}[\tilde{u}] M_{N+1}[\psi_j] + M_{N/2+2}[\tilde{u}]
\|\la x\ra^{3/4} \Gamma^{\le N}\Box\psi_j\|_{L^2L^2}.
\end{multline*}
And from \eqref{abstractGoal2}, we obtain
\begin{multline*}
  \la t\ra^{-\delta} \|r^{-1} \Gamma^{\le N/2} \partial \phi \:\Gamma^{\le
  N} \partial^{\le 1} \tilde{u} \|_{L^1L^1L^2} 
+
\|r^{-\frac{5}{4}} \Gamma^{\le N/2} \partial \phi \:\Gamma^{\le
  N} \partial^{\le 1} \tilde{u} \|_{L^1L^1L^2} 
\\\lesssim M_N[\tilde{u}] M_{N/2+3}[\phi] + M_N[\tilde{u}] \|\la
  x\ra^{3/4}\Gamma^{\le N/2+2} \Box \phi\|_{L^2L^2},
\end{multline*}
\begin{multline*}
  \la t\ra^{-\delta} \|r^{-1} \Gamma^{\le
  N/2} \partial \tilde{u}\:\Gamma^{\le N} \partial^{\le 1}\phi \|_{L^1L^1L^2} 
+
\|r^{-\frac{5}{4}} \Gamma^{\le
  N/2} \partial \tilde{u} \:\Gamma^{\le N} \partial^{\le 1}\phi
\|_{L^1L^1L^2}
\\\lesssim M_N[\phi] M_{N/2+3}[\tilde{u}] + M_N[\phi] \|\la
x\ra^{3/4}\Gamma^{\le N/2+2}\Box \tilde{u}\|_{L^2L^2},
\end{multline*}
\begin{multline*}
  \la t\ra^{-\delta} \|r^{-1} \Gamma^{\le
  N/2+1} \partial \psi_j \:\Gamma^{\le N} \partial^{\le 1}\tilde{u}\|_{L^1L^1L^2} 
+
\|r^{-\frac{5}{4}} \Gamma^{\le
  N/2+1} \partial \psi_j \:\Gamma^{\le N} \partial^{\le 1}\tilde{u}
\|_{L^1L^1L^2}
\\\lesssim M_N[\tilde{u}] M_{N/2+4}[\psi_j] + M_N[\tilde{u}] \|\la
x\ra^{3/4} \Gamma^{\le N/2+3}\Box\psi_j\|_{L^2L^2},
\end{multline*}
\begin{multline*}
  \la t\ra^{-\delta} \|r^{-1} \Gamma^{\le N/2} \partial \phi \:\Gamma^{\le
  N} \partial u \|_{L^1L^1L^2} 
+
\|r^{-\frac{5}{4}} \Gamma^{\le N/2} \partial \phi \:\Gamma^{\le
  N} \partial u \|_{L^1L^1L^2}
\\\lesssim M_N[u] M_{N/2+3}[\phi] + M_N[u] \|\la x\ra^{3/4}
\Gamma^{\le N/2+2}\Box \phi\|_{L^2L^2},
\end{multline*}
and
\begin{multline*}
  \la t\ra^{-\delta} \|r^{-1} \Gamma^{\le
  N/2+1} \partial u\:\Gamma^{\le N} \partial^{\le 1}\phi \|_{L^1L^1L^2} 
+
\|r^{-\frac{5}{4}} \Gamma^{\le
  N/2+1} \partial u \:\Gamma^{\le N} \partial^{\le 1} \phi
\|_{L^1L^1L^2}
\\\lesssim M_N[\phi] M_{N/2+4}[u] + M_N[\phi]\|\la x\ra^{3/4}
\Gamma^{\le N/2+3} \Box u\|_{L^2L^2}.
\end{multline*}

In order to control the remaining term
\begin{equation}
  \label{remaining}
\la t\ra^{-\delta} \|r^{-1} \Gamma^{\le N/2} \phi \:\Gamma^{\le
  N-1} \partial^2 u \|_{L^1L^1L^2} 
+
\|r^{-\frac{5}{4}} \Gamma^{\le N/2} \phi \:\Gamma^{\le
  N-1} \partial^2 u \|_{L^1L^1L^2},
\end{equation}
we shall argue as in
\eqref{abstractGoal}, but it will be necessary to use a perturbation
of $\Box$ in \eqref{mtt_lem} so as to not exceed the allowable regularity.
It is likely that estimates such
as \cite[Lemma 3.1]{KlSid} could be used as an alternative to
\eqref{mtt_lem} for these cases, but we do not explore such here.

We again consider $|x|<c_1 s/8$ and $|x|\ge c_1s/8$ separately.  
Sobolev embeddings as above give that \eqref{remaining} is
\begin{multline}\label{dividedAnalysis2}
\lesssim
  \sum_{j<\log t} \sum_{k<j-\tilde{C}} 2^{-k} \|\Gamma^{\le N/2+2} \phi\|_{L^2_{2^j}
    L^2_{2^k}} \|\Gamma^{N-1} \partial^2 u\|_{L^2_{2^j}L^2_{2^k}}
\\+ \la t\ra^{-\delta} \int_0^t \int_{r\ge c_1s/4} \la r\ra^{-1}
\|\Gamma^{\le N/2+
  2} \phi(t,r\cdot)\|_{L^2(\S^3)} \|\Gamma^{\le N} \partial u(t,r\cdot)\|_{L^2(\S^3)} r^2 dr ds
\\+ \int_0^t \int_{r\ge c_1s/4} \la r\ra^{-5/4} \|\Gamma^{\le N/2+
  2} \phi(t,r\cdot)\|_{L^2(\S^3)} \|\Gamma^{\le N} \partial u(t,r\cdot)\|_{L^2(\S^3)} r^2 dr ds
\end{multline}
with $\tilde{C}>0$ as above.  Using the Schwarz inequality and arguing
as in the proof of \eqref{abstractGoal}, the last
two terms are $\O(M_{N/2+2}[\phi](t) M_N[u](t))$.

It remains to bound the first term of \eqref{dividedAnalysis2}.  Here
we again set
$h^{IK,\alpha\beta}=-A^{\alpha\beta}_{IJK}\phi^J-B^{\gamma\alpha\beta}_{IJK}\partial_\gamma \phi^J$.
Using \eqref{boxhvf} and \eqref{weighted-Sob}, we have
\begin{multline*}
  2^k \|\Box_h \Gamma^{\le N} u\|_{L^2_{2^j}L^2_{2^k}} \lesssim
  2^{-k/2} \Bigl(\|\Gamma^{\le N/2+4} \phi\|_{L^2_{2^j}L^2_{2^k}}
  \|\Gamma^{\le N} \partial\tilde{u}\|_{L^\infty L^2}
\\+ \|\Gamma^{\le N/2+3}\partial\tilde{u}\|_{L^\infty L^2} \|\Gamma^{\le
  N}\partial^{\le 1}\phi\|_{L^2_{2^j}L^2_{2^k}}
+ \|\Gamma^{\le N/2+3} \partial\psi_1\|_{L^\infty L^2} \|\Gamma^{\le
  N} \partial^{\le 1} \tilde{u}\|_{L^2_{2^j}L^2_{2^k}} \\+ \|\Gamma^{\le
  N/2+3}\partial^{\le 1}\tilde{u}\|_{L^2_{2^j}L^2_{2^k}} \|\Gamma^{\le
  N}\partial \psi_1\|_{L^\infty L^2} +
\|\Gamma^{N/2+4}\phi\|_{L^2_{2^j}L^2_{2^k}} \|\Gamma^{\le N}\partial
u\|_{L^\infty L^2} 
\\+\|\Gamma^{N/2+4}\partial u\|_{L^\infty L^2} \|\Gamma^{\le
  N}\partial^{\le 1} \phi\|_{L^2_{2^j}L^2_{2^k}}
+ \|\Gamma^{\le N/2+3}\partial^{\le 1}\tilde{u}\|_{L^2_{2^j}L^2_{2^k}}
\|\Gamma^{\le N+1}\partial\psi_2\|_{L^\infty L^2} 
\\+ \|\Gamma^{\le N/2+4}\partial\psi_2\|_{L^\infty L^2} \|\Gamma^{\le
  N} \partial^{\le 1} \tilde{u}\|_{L^2_{2^j}L^2_{2^k}}\Bigr).
\end{multline*}
By \eqref{mtt_lem}, for $k<j-\tilde{C}$, we have
\[\|\Gamma^{\le N-1} \partial^2 u\|_{L^2_{2^j}L^2_{2^k}} \lesssim 2^{-k} \|\Gamma^{\le N}
 \partial  u\|_{L^2_{2^j}L^2_{2^k}} + 2^k \|\Box_h \Gamma^{\le N} u\|_{L^2_{2^j}L^2_{2^k}}.\]
Using the previous two estimates, it thus follows that
\begin{multline*}
    \sum_{j<\log t} \sum_{k<j-\tilde{C}} 2^{-k} \|\Gamma^{\le N/2+ 2} \phi\|_{L^2_{2^j}
    L^2_{2^k}} \|\Gamma^{\le N-1}\partial^2 u\|_{L^2_{2^j}L^2_{2^k}}
\lesssim M_{N/2+2}[\phi] M_N[u] \\+
    M_{N/2+2}[\phi]\Bigl( \Bigl[M_{N/2+4}[\phi] + M_{N/2+3}[\psi_1] + M_{N/2+4}[\psi_2]\Bigr]M_N[\tilde{u}]
    \\+M_{N/2+3}[\tilde{u}] \Bigl[M_N[\phi] + M_N[\psi_1] + M_{N+1}[\psi_2]\Bigr] \\
    + M_{N/2+4}[\phi] M_N[u]  +
    M_{N/2+4}[u] M_N[\phi]\Bigr)
\end{multline*}
by splitting the decay equally amongst the $L^2_{2^j} L^2_{2^k}$
portions and using the Cauchy-Schwarz inequality to sum.  This
completes the proof.\qed

\bibliography{exterior}

\end{document}